\documentclass[11pt,a4paper]{amsart}

\usepackage{verbatim}
\usepackage{setspace}
\usepackage[left=2.5cm,right=2.5cm,top=2.5cm,bottom=3cm,includeheadfoot]{geometry}
\usepackage{latexsym}
\usepackage{amssymb}
\usepackage{amsthm}
\usepackage{amsmath}
\usepackage{tikz}
\usetikzlibrary{matrix}
\usetikzlibrary{positioning}
\usetikzlibrary{arrows}
\usepackage{graphicx}
\usepackage[font=small,labelfont=bf]{caption}

\usepackage[all]{xy}

\usepackage{caption}

\usepackage{setspace}

\usepackage{enumitem}

\usepackage{amsfonts}
\usepackage[english]{babel}
\usepackage[latin1]{inputenc}

\numberwithin{figure}{section}
\numberwithin{table}{section}
\numberwithin{equation}{section}
\swapnumbers

\newtheorem{theorem}{Theorem}[section]
\newtheorem{lemma}[theorem]{Lemma}

\theoremstyle{definition}
\newtheorem{definition}[theorem]{Definition}
\newtheorem{example}[theorem]{Example}

\newtheorem{proposition}[theorem]{Proposition}
\theoremstyle{definition}
\newtheorem{remark}[theorem]{Remark}

\newtheorem{corollary}[theorem]{Corollary}

\newtheorem{question}[theorem]{Question}

\newcommand{\field}[1]{\mathbb{#1}}
\newcommand{\C}{\field{C}}
\newcommand{\R}{\field{R}}
\newcommand{\N}{\field{N}}
\newcommand{\Z}{\field{Z}}

\newcommand{\fh}{\mathfrak {h}}
\newcommand{\ft}{\mathfrak {t}}

\newcommand{\tham}{(M,\omega,\phi)}
\newcommand{\tka}{(M,\omega,J,\phi)}

\newcommand{\Mcore}{M_{\operatorname{core}}}

\newcommand{\Mexc}{M_{\operatorname{exc}}}
\newcommand{\Mtall}
{M_{\operatorname{tall}}}

\newcommand{\barphi}
{\bar{\phi}}

\newcommand{\tall}{{\operatorname{tall}}}

\def \CP {{\mathbb C}{\mathbb P}}

\newcommand{\mute}[1] {}

\begin{document}
\title[K\"ahler complexity one Hamiltonian $T$-manifolds  have trivial paintings]{\protect\boldmath K\"ahler complexity one Hamiltonian $T$-manifolds have trivial paintings}

\author{Isabelle Charton}
\address{School of Mathematics, Tel Aviv University, Israel}
\email{icharton@math.haifa.ac.il}

\author{Liat Kessler} 
\address{Department of Mathematics, Physics, and Computer Science, University of Haifa, Israel}
\email{lkessler@math.haifa.ac.il}

\author{Susan Tolman}
\address{Department of Mathematics,
University of Illinois,
Urbana, IL 61801, USA}
\email{tolman@illinois.edu}

	\begin{abstract}
 Let a torus $T$ act on a symplectic manifold $(M,\omega)$ with moment map $\phi$.
We say that the Hamiltonian $T$-manifold $(M,\omega,\phi)$ has complexity one if $\frac{1}{2} \dim M - \dim T = 1$, and that it is K\"ahler if it admits an invariant compatible complex structure.
In this paper, we show how the class of K\"ahler complexity one Hamiltonian $T$-manifolds sits
inside the class of complexity one Hamiltonian $T$-manifolds by proving that every compact, connected K\"ahler complexity one Hamiltonian $T$-manifold  has a trivial painting.
As a corollary, we show that  two tall compact, connected K\"ahler  complexity one Hamiltonian $T$-manifolds
are symplectomorphic exactly if they have the same genus,
Duistermaat-Heckman measure, and skeleton.
Here,  $(M,\omega,\phi)$ is tall exactly if every non-empty fiber $\phi^{-1}(\alpha)$ contains more than one orbit.
\end{abstract}

\subjclass[2010]{Primary 53D35; Secondary 53D20,57S15 }

\maketitle

\section{Introduction}

Once and for all, fix a compact torus $T$  with Lie algebra $\mathfrak{t}$.
Let $M$ be a manifold and $\omega$ be a {\bf symplectic form} on $M$,
that is,
a closed, non-degenerate $2$-form $\omega \in \Omega^2(M)$.
A $T$-action\footnote{Unless specified otherwise, we always assume that our actions are effective.} on  $(M,\omega)$  is 
\textbf{Hamiltonian} if there 
exists a \textbf{moment map}, i.e., a  map $\phi \colon M \rightarrow \mathfrak{t}^*$  such that
\begin{align*}
\iota_{X_\zeta}\omega = -\operatorname{d}\left\langle \phi ,\zeta \right\rangle\quad \text{for all $\zeta \in \ft$.}
\end{align*} 
 Here, $X_{\zeta}$ is the vector field generated by $\zeta$ and $\left\langle 
\cdot,\cdot\right\rangle $ is the natural pairing between $\mathfrak{t}^*$ and $\mathfrak{t}$. In this case, we  call $\tham$ a \textbf{Hamiltonian $\mathbf T$-manifold}.
The \textbf{complexity} of $\tham$ is the non-negative integer
$$k:= \textstyle \frac{1}{2}\operatorname{dim}(M)-\operatorname{dim}(T).$$

A \textbf{complex structure} on $M$ is 
an automorphism $J \colon TM \to TM$ with $J^2 = - \operatorname{Id}$ that is induced from a holomorphic atlas on $M$.  A complex structure $J$ is \textbf{compatible} with the symplectic form $\omega$ if $\omega(\cdot,J \cdot)$ is a Riemannian metric on $M$;  $(M,\omega,J)$ is then a {\bf K\"ahler manifold}.   For us, a {\bf K\"ahler Hamiltonian $\mathbf T$-manifold} is a Hamiltonian
$T$-manifold $(M,\omega,\phi)$ equipped with a compatible $T$-invariant complex structure $J$.

In 1976, Thurston constructed the first example of a symplectic manifold with no K\"ahler structure \cite{Thurston}.
Since then, many of the important questions in symplectic geometry have arisen from exploring the similarities and differences between K\"ahler and symplectic geometry.
One aspect of this program is adapting powerful theorems from K\"ahler geometry so that they still hold in the symplectic category.  Another aspect is determining which symplectic manifolds admit K\"ahler structures; see, e.g., \cite{Gompf,McDuff}. 
In this paper, and a proposed follow-up, we seek to determine which
complexity one Hamiltonian $T$-manifolds admit K\"ahler structures.

It follows from Delzant's classification of compact complexity zero Hamiltonian $T$-manifolds  that any such manifold admits a $T$-invariant complex structure that is compatible with the symplectic form \cite{Delzant}.
Karshon   has shown that this property holds for
 any compact complexity one Hamiltonian $T$-manifold of 
dimension four \cite[Theorem 7.1]{Karshon4}.  However, Tolman gave an example of a six-dimensional complexity one Hamiltonian $T$-manifold with only isolated fixed points that does not admit a compatible $T$-invariant complex structure \cite{TolmanExample}; see also \cite{Lerman}.

In this paper, we give a necessary condition for the existence of a $T$-invariant compatible complex structure for complexity one Hamiltonian $T$-manifolds of any dimension, in terms of one of the discrete invariants of complexity one spaces introduced in \cite{Tall2, Tall3}. We conjecture that this condition is also sufficient. We plan to explore this question in a subsequent paper using the classification of complexity one  $T^\C$-actions on algebraic varieties in \cite{AH,AHS} and \cite{Timashev}. 
To state our theorem precisely, we recall some invariants from Karshon-Tolman's analysis of complexity one Hamiltonian $T$-manifolds.

Let $\tham$ be a complexity one Hamiltonian $T$-manifold.  Given $p \in M$,
we say that $p$ is {\bf short} if
$[p]$ is an isolated point in the reduced space $\phi^{-1}(\phi(p))/T$,
and {\bf tall} otherwise.
Additionally, we say that $p \in M$ is {\bf exceptional} if every point in some neighborhood
of $[p]$ in $\phi^{-1}(\phi(p))/T$ has a strictly smaller stabilizer.
By definition,  every short point is exceptional.
Let $M_\tall \subseteq M$ and $\Mexc \subseteq M$ be the set of tall and exceptional points in $M$, respectively. The quotient $(\Mexc \cap M_\tall)/T$ is the  {\bf skeleton} of $M$.

Now assume that $M$ is compact and connected. 
By \cite[Proposition 1.2]{Tall3}, 
 there exists  a {\bf genus} $g \in \N := \Z_{\geq 0}$  so that for each  $\alpha \in \phi(M_\tall)$  the reduced space $\phi^{-1}(\alpha)/T$ is a compact  oriented surface of genus $g$; fix such a surface $\Sigma$.  
A \textbf{painting} is a map $f \colon (\Mexc \cap M_\tall)/T\rightarrow \Sigma$  such that
$$(\bar{\phi}, f) \colon (\Mexc \cap M_\tall)/T \longrightarrow \phi(M_\tall) \times \Sigma $$
is injective, where $\bar{\phi} \colon M/T \to \ft^*$  is the orbital moment map. Paintings $f$ and $f'$ are {\bf equivalent} if there exists an orientation-preserving homeomorphism  $\nu \colon \Sigma \to \Sigma$ such that $f$  and $\nu \circ f'$ are homotopic through paintings. 
Moreover, a  painting $f$ is {\bf associated to {$\mathbf M$}} 
if  it is the restriction to
$(\Mexc \cap M_\tall)/T$ of  
 a map $F \colon M_{\tall}/T \rightarrow \Sigma$  that induces a homeomorphism
$$(\bar{\phi}, F) \colon M_\tall/T  \longrightarrow \phi(M_\tall)\times \Sigma, $$  
and that preserves the orientation on each reduced space $\phi^{-1}(\alpha)/T$.
By \cite[Proposition 1.2]{Tall3}, the set of paintings associated to $M$ is non-empty, and all such paintings are equivalent.
Finally, a painting is {\bf trivial} if it is equivalent to a locally constant painting.
We can now state our main theorem.

\begin{theorem}\label{MainThm}
Let $\tham$ be a compact, connected complexity one Hamiltonian $T$-manifold. If $M$ admits a $T$-invariant complex structure that is compatible 
with the symplectic form $\omega$, then the paintings associated to $M$ are trivial.
\end{theorem}

\begin{remark}\label{rem1}
There are many examples of complexity one spaces with non-trivial paintings. For example, if the orbital moment map $\overline{\phi} \colon S_0 \to \ft^*$ is not injective for some connected component $S_0$ of the skeleton $(\Mexc \cap M_\tall)/T$,
 then the painting cannot be trivial.  This holds, in particular, for the complexity one Hamiltonian $T$-manifold constructed in \cite{TolmanExample}, and so Theorem~\ref{MainThm} gives an alternate proof that it admits no compatible $T$-invariant complex structure.  See also Remark~\ref{examples}. 
\end{remark}

 To explain the implications of Theorem~\ref{MainThm} we recall Karshon and Tolman's classification of complexity one Hamiltonian manifolds. 
Let $(M,\omega,\phi)$ and $(M',\omega',\phi')$ be compact, connected complexity one
Hamiltonian $T$-manifolds. A
 homeomorphism $g \colon (\Mexc \cap M_\tall)/T \to (\Mexc' \cap M'_\tall)/T$
 which  intertwines the orbital moment
maps $\overline{\phi}$ and $\overline{\phi'}$ is an {\bf isomorphism of skeletons} if it
 takes every orbit in $(\Mexc \cap M_\tall)/T$ to an orbit in $(\Mexc' \cap M'_\tall)/T$ with the same stabilizer and symplectic slice representation. (Recall that the symplectic slice representation at  $p \in M$ 
is the representation of the stabilizer $H$ of $p$ on the symplectic vector space $T_p {{(T \cdot p)}^\omega}/T_p {(T \cdot p)}$; see Definition \ref{def:sym-slice}.)
In \cite{Tall1,Tall2, Tall3}, Karshon and Tolman  prove that if $M$ and $M'$ are tall, then  there is an equivariant symplectomorphism from $M$ to $M'$ that intertwines
the moment maps if and only if they have the same genus and the same Duistermaat-Heckman measure,
and there is an isomorphism of skeletons $g \colon (\Mexc \cap M_\tall)/T \to (\Mexc' \cap M'_\tall)/T$ such that the pull-back of a painting associated to $M'$
is a painting associated to $M$.
Here,  {\bf Duistermaat-Heckman measure} is the measure on $\mathfrak{t}^{*}$ obtained as the push-forward by the moment map of the {\bf Liouville measure} on $M$, which  itself is given by integrating the volume
form ${\omega^n}/{n!}$ with respect to the symplectic orientation. 
Therefore, Theorem~\ref{MainThm} has the following corollary.

\begin{corollary}\label{MainCor}
Let $(M,\omega,\phi)$ and $(M',\omega',\phi')$ be compact, connected tall complexity one
Hamiltonian $T$-manifolds.
If $M$ and $M'$ both admit compatible $T$-invariant complex structures, then $M$ and $M'$
are equivariantly symplectomorphic if and only if they have the same genus and Duistermaat-Heckman measure, and isomorphic 
skeletons.
\end{corollary}

\begin{remark}\label{examples} In general, complexity one Hamiltonian $T$-manifolds with the same genus and Duistermaat-Heckman measure and isomorphic skeletons can have non-equivalent paintings.
 For example, for any $g > 0$,  in \cite[Example 1.11]{Tall3} the authors construct a countably infinite family of six-dimensional tall complexity one Hamiltonian $T$-manifolds of genus $g$ with the same Duistermaat-Heckman measure and isomorphic skeletons but non-equivalent paintings. Similarly,  in \cite[Example 1.12]{Tall3} they construct a countably infinite family of $8$-dimensional  tall complexity one Hamiltonian $T$-manifolds of genus $0$  with the same properties.  By Corollary~\ref{MainCor}, at most one of the Hamiltonian $T$-manifolds in each family admits a compatible invariant complex structure.
 
\end{remark}

Karshon and Tolman have conjectured that their classification can be generalized as follows:
If $(M,\omega,\phi)$ and $(M',\omega',\phi')$ are any compact, connected complexity one $T$-spaces, then  there is an equivariant symplectomorphism from $M$ to $M'$ that intertwines
the moment maps if and only if they have the same genus and the same Duistermaat-Heckman measure,
and there is an isomorphism from the skeleton of $M$ to the skeleton of $M'$ such that the pull-back of a painting associated to $M'$
is a painting associated to $M$. If this conjecture is true, then Corollary~\ref{MainCor} also extends to the case that $M$ and $M'$ are not tall.

It is natural to ask if the converse to Theorem~\ref{MainThm} is true.

\begin{question}\label{Question}
Let $\tham$ be a compact, connected complexity one Hamiltonian $T$-manifold. Assume that the paintings associated to $\tham$ are trivial. Does $M$ necessarily admit a $T$-invariant complex structure that is compatible with the symplectic form $\omega$? 
\end{question}

 In \cite{LPT}, Liu, Palmer, and Tolman  prove that a tall, compact, simply connected complexity one Hamiltonian $T$-manifold $\tham$
extends to a complexity zero torus action exactly if the skeleton is  ``two-colorable" and the associated painting is trivial.   
By \cite{Delzant}, this implies that the answer to Question~\ref{Question} is affirmative in this setting. 
By \cite[Theorem 7.1]{Karshon4}, the answer is affirmative
(and the painting is always trivial) in dimension four.
We plan to address Question~\ref{Question} more generally in our next paper, in part by detailing the relationship between the classification of Hamiltonian $T$-manifolds in \cite{Tall1, Tall2, Tall3} and the  classification of
 complexity one  $T^\C$-actions on algebraic varieties  in \cite{AH,AHS} and \cite{Timashev}.

The structure of the paper is as follows. In Sections \ref{sec:2} and \ref{ss:holomorphic}, we give the required preliminaries on complexity one Hamiltonian torus actions and on K\"ahler torus actions, respectively. 
 Given a compact K\"ahler complexity one Hamiltonian $T$-manifold
$\tka$, let $\Mcore$ be the union of all the $T^\C$-orbits whose closures contain tall exceptional points. In Section~\ref{sec:mcore}, we prove that
the orbital moment map restricts to a local homeomorphism
\begin{equation*} 
    \bar{\phi} \colon {(\overline{\Mcore} \cap \Mtall)}/T \to \phi(\Mtall).
\end{equation*} 
In Section \ref{sec:proof}, we prove that this implies Theorem \ref{MainThm}.

\subsection*{Acknowledgements}
The authors thank Yael Karshon, Reyer Sjammar, and Hendrik S\"uss
for answering our many questions.
The second author was supported in part by 
an NSF-BSF Grant 2021730. 
The third author was supported in part by Simons Foundation, Grant 637996, and
by the National Science Foundation under Grant
DMS-2204359.

\section{Complexity one Hamiltonian torus actions}\label{sec:2}

In this section, we introduce the facts about complexity one Hamiltonian torus actions that we will need in our argument.
We begin with a brief review of Hamiltonian torus actions (of arbitrary complexity);
 recall that $T$ is a compact torus.

\begin{definition}\label{def:sym-slice} 
Let $(M,\omega,\phi)$ be a Hamiltonian $T$-manifold.
Fix $p \in M$ and let $H$ be the stabilizer of $p$ in $T$ and
let $\mathfrak O := T \cdot p$ be the $T$-orbit through $p$. 
 Then $T_p \mathfrak O$ is isotropic and so the {\bf  symplectic slice} $T_p \mathfrak O^\omega/T_p \mathfrak O$ inherits a natural symplectic structure. 
The $H$-action on $T_p M$ induces a symplectic representation
on  $T_p \mathfrak O^\omega/T_p \mathfrak O$,
called the {\bf (symplectic) slice representation}.
 By choosing an invariant compatible complex structure, we may
think of the slice representation as a complex representation;  up to isomorphism, the complex representation does not depend on the choice of such structure.
The slice representation is isomorphic to the $H$-action on $\C^{k}$ given by an inclusion
\begin{align*}
    \rho=(\rho_1,\dots, \rho_{k})\colon  
     H \hookrightarrow (S^1)^{k},
    \end{align*}
where $k = \frac{1}{2} \dim M - \dim T + \dim H.$
The weights of this representation,  which we call the {\bf isotropy weights} at $p$, are
the  differentials at the identity of $H$ of the components $\rho_i\colon H \rightarrow S^1$. 
\end{definition}

From now on we fix an  inner product on the Lie algebra $\mathfrak t$,
and use it to fix a canonical identification of $\ft^*$ with
$\fh^\circ \oplus \fh^*$ for any subspace $\fh \subseteq \ft$, where
$\fh^\circ \subseteq \ft^*$ is the annihilator of $\fh$.
Given a closed   subgroup $H \subseteq T$ and an inclusion $\rho \colon H \hookrightarrow (S^1)^k$ for some $k \in \N$, 
set $Y := T \times _H \mathfrak{h}^\circ \times  \C^{k}$, where $H$ acts on $T \times \mathfrak{h}^\circ \times \C^k$ by $h \cdot (t,\alpha, z) = (t h^{-1}, \alpha, \rho(h) z)$. The manifold  $Y$
can be endowed with a canonical symplectic form $\omega_Y$  so that the $T$-action on $Y$ given by
$s\cdot [t, \alpha, z]= [st,\alpha,z]$ is Hamiltonian with homogeneous moment map
\begin{equation}\label{eq:phiY}
\phi_Y \colon Y \to \ft^*,\,\,\, [t,\alpha,z] \mapsto \alpha + \sum \frac{1}{2} \eta_i |z_i|^2. 
\end{equation}
Given $p\in M$ with stabilizer $H$ and symplectic slice representation $\rho$, we call $(Y, \omega_Y, \phi_Y)$
the {\bf local model at $\mathbf p$}.  
By the {\bf (symplectic) local normal form theorem} below,  the manifold near a point $p$ is determined by the local model at $p$ \cite{GSNormalForm, Marle}.

\begin{theorem}[Guillemin-Sternberg, Marle] \label{thm:lnf}
 Let $(M,\omega,\phi)$ be a Hamiltonian $T$-manifold.
Given $p\in M$, let $(Y, \omega_Y, \phi_Y)$ be the local model at $p$.
There exists a $T$-equivariant symplectomorphism from
a $T$-invariant neighborhood of $p$ in $M$ to a neighborhood of $[1,0,0]$ in $Y$
that takes $p$ to $[1,0,0]$.

\end{theorem}

\begin{definition}\label{def:moment cone}
Let $(M,\omega,\phi)$ be a Hamiltonian $T$-manifold. Given a point $p \in M$ with stabilizer $H$ and isotropy weights $\eta_1, \dots \eta_k$; let $(Y, \omega_Y, \Phi_Y)$ be the local model at $p$. 
The {\bf moment cone at $\mathbf p$} is the closed polyhedral cone 
$$ \overline{C_p} :=  \phi(p) + \Phi_Y(Y) =  \phi(p) + \fh^{\circ} +  \sum_{i=1}^{k} \R_{\geq 0} \eta_i.$$
 As the notation suggests, the moment cone $\overline{C_p}$ is the closure of its relative interior
$$ C_p :=  \phi(p) + \fh^{\circ} +  \sum_{i=1}^{k} \R_{> 0} \eta_i.$$
\end{definition}

The symplectic local normal form theorem is a key ingredient in the theorem below,
which  is a special case of Sjamaar's theorem \cite[Theorem 6.5]{Sjamaar};
see also \cite[Theorems 6.1 and 6.2]{LMTW}.

\begin{theorem}\label{le:imagecone}  Let $(M,\omega,\phi)$ be a Hamiltonian $T$-manifold  with proper moment map.
Given a point $p \in M$ with moment cone $\overline{C_p}$,
there exists a neighborhood $V$ of $\phi(p)$ in $\ft^*$ such  that
$$\phi(M) \cap V = 
 \overline{C_p}  \cap V.$$
\end{theorem}

 Before restricting our focus to complexity one Hamiltonian torus actions,
we recall that in the complexity zero case the orbital moment map is locally a homeomorphism onto its image. 
\begin{lemma}\label{le:eqdarboux}
 Let $(M,\omega,\phi)$ be a complexity zero Hamiltonian $T$-manifold.  Given $p \in M$,
there exists a $T$-invariant neighborhood $W$ of $p$  such that $\phi$ induces a homeomorphism
from $W/T$ to   an open neighborhood of $\phi(p)$ in the moment cone $\overline{C_p}$.
Moreover, an orbit in $W$ maps to the relative interior $C_p$
of $\overline{C_p}$ exactly if the orbit has trivial stabilizer.

\end{lemma}

\begin{proof} 
Let  $p$ have  stabilizer $H$ and isotropy weights $\eta_1,\dots,\eta_h \in \fh^*$.
By the symplectic local normal form theorem, we may assume that $M$ is a local model
$Y = T \times_H \fh^\circ \times \C^{h}$ with  moment map $\phi_Y \colon Y \to \ft^*$, $[t,\alpha,z] \mapsto  \alpha + \sum \frac{1}{2} \eta_i |z_i|^2$. Moreover,  by a dimension count, the weights $\eta_1,\dots,\eta_h$ are a basis for $\fh^*$, and  $H$ acts on $\C^h$ through an
isomorphism with $(S^1)^h$. Hence, the quotient $Y/T$ is naturally
homeomorphic to $\fh^\circ \times \R_{\geq 0}^h$. The first
claim follows immediately from the fact that the map from $\fh^\circ \times \R^h$ to $\ft^*$ sending $(\alpha, x)$ to $\alpha + \sum_{i=1}^h x_i \eta_i$ is a linear isomorphism.
Finally,  $\phi_Y([t,\alpha, z])$ lies in the relative interior $C_p$ of $\overline{C_p}$ exactly if $z_i \neq 0$ for all $i$,
that is, exactly if $[t,\alpha,z]$ has trivial stabilizer.

\end{proof}

The next lemma, which gives a criterion for a point in a complexity one Hamiltonian $T$-manifold to be tall,
follows immediately from 
\cite[Lemmas 5.4 and 5.8]{Tall1}.

\begin{lemma}[Karshon-Tolman]\label{lem:KT-pos-combo}
Let $\tham$ be a complexity one Hamiltonian $T$-manifold. 
Let  $p \in M$ be a  point with  stabilizer $H$ and slice representation $\rho \colon H \to (S^1)^{h+1}$.
There
exists a (unique up to sign)  non-zero $\xi \in  \Z^{h+1}$ such that the following sequence is exact:
$$1 \to H \stackrel{\rho}{\to} (S^1)^{h+1} \stackrel{P}{\to} S^1 \to 1,$$
where $P(z) = \prod_{i=0}^h z_i^{\xi_i}.$
Moreover, we can choose $\xi$ so that $\xi_i \geq 0$ for all $i$ exactly if $p$ is tall.
\end{lemma}

\begin{remark}
\label{rem:xii}
In the situation of Lemma~\ref{lem:KT-pos-combo}, let $\eta_0,\dots, \eta_h$ be the isotropy weights for $p$.
Differentiating the identity $P \circ \rho = 1$ yields $\sum_i \xi_i \eta_i = 0$.  Since the isotopy weights span the $h$-dimensional vector space $\fh^*$, this implies
that if $\sum_{i=0}^h x_i \eta_i = 0$, then $x=(x_0,\ldots,x_{h})$ is multiple of $\xi=(\xi_0,\ldots,\xi_h)$.
In particular, if $x \neq 0$  then  $x_i x_j \geq 0$ for all $i, j$  exactly if $p$ is tall.

    \end{remark}

Using this criterion,  the moment cone at a tall point can be presented as a union of pairwise disjoint cones associated to linearly independent subsets of the set of isotropy weights at $p$.

\begin{lemma}\label{cor:linear}
Let $\tham$ be a complexity one Hamiltonian $T$-manifold.
Consider a tall point $p \in M$ with  stabilizer $H$,  isotropy weights $\eta_0,\dots,\eta_h \in \fh^*$,  and moment cone $\overline{C_p}$.
Let $\mathcal I$ be the set of $I \subseteq \{0,\dots,h\} $ such that $\{ \eta_i\}_{i \in I}$ is linearly independent. Given $I \in \mathcal I$, let $C_I 
:=\phi(p) + \mathfrak h^\circ +  \sum_{i \in I}  \R_{>0 } \,  \eta_i$.
Then
$$ \overline{C_{p}} = \bigcup_{I \in \mathcal I} C_I \quad \text{and} \quad C_I \cap C_J = \emptyset \text{ if } I \neq J \in \mathcal I.$$

\end{lemma}

\begin{proof}
First, we will show that 
$\overline{C_p} =\cup_{I \in \mathcal I} C_I$, which is a variant of Cartheardory's theorem for cones.
If not, then there exists $\alpha \in \sum_{i=0}^h \R_{\geq 0} \eta_i$ such
that $\alpha \not\in \sum_{i \in I} \R_{\geq 0} \eta_i$ for any $I \in \mathcal I$.
Fix  $I \subseteq \{0,\dots,h\}$ such that 
$\alpha = \sum_{i \in I} x_i \eta_i$ for some   $x \in \R^{I}_{\geq 0}$
but
$\alpha \not\in \sum_{j \in J} \R_{\geq 0} \eta_j$ for any $J \subsetneq I$.
By assumption $I \not\in \mathcal I$ and so there  exists non-zero $y \in \R^I$ such that $\sum_{i \in I} y_i \eta_i = 0$; we may further assume that $y_i < 0$ for some $i \in I$.  
Then 
 \[
  \alpha = \sum_{i \in I} x_i\eta_i + t \left( \sum_{i \in I}  y_i \eta_i \right) = \sum_{i \in I}(x_i+t  y_i)\eta_i
 \]
 for all $t \in \R$.
 Let $t_0$ be the smallest non-negative number such that $x_{i_0}+t_0 y_{i_0}=0$ for some $i_0$; such a $t_0$ must exist since all $x_i$ are non-negative and at least one of the $y_i$ is negative.
 Then $\alpha$ is in the non-negative span of $\{\eta_i\}_{i \in I \smallsetminus \{i_0\}}$,
 which is impossible.

Next, we will show that $C_I \cap C_{J} = \emptyset$ for all $I \neq J \in \mathcal I$.
If not, then there exist   $I \neq  J \in \mathcal I$
and positive numbers $x_i$ for $i \in I$ and $y_j$ for $j \in J$
such that \begin{equation} \label{eqdiff}
\sum_{i \in I} x_i \eta_i - \sum_{j \in J} y_j \eta_j = 0.
\end{equation}
Since  $\{\eta_i\}_{i \in I}$ and $\{\eta_j\}_{j \in J}$ are linearly independent, the fact that $I \neq J$ implies that
$I \not\subseteq J$ and $J \not\subseteq I$.
Hence, \eqref{eqdiff} has both positive and negative  coefficients, which  contradicts 
Remark~\ref{rem:xii}.
\end{proof}

By the Convexity theorem of Atiyah \cite{Atiyah} and Guillemin-Sternberg \cite{GSconvex}, the {\bf moment image} $\phi(M)$  of a compact, connected Hamiltonian $T$-manifold $\tham$ is a convex polytope.
In the complexity one case the moment image of $M_\tall$ is also convex.
\begin{lemma}[\cite{Tall4}, Lemma 2.9]
\label{deltatall}
	Let $\tham$ be a compact, connected  complexity one Hamiltonian $T$-manifold. Then $\phi(M_\tall) \subseteq \ft^*$
    is convex.
\end{lemma}

The next lemma will be used together with Lemma~\ref{closure} below to help understand orbit-closures for actions of complex tori.

\begin{lemma} \label{lem:nu}
Let $\tham$ be a complexity one Hamiltonian $T$-manifold.
Consider a point $p \in M$ with stabilizer  $H \subseteq T$ and isotropy  weights 
$\eta_0,\dots, \eta_h \in \fh^*$.
Given $I \subseteq \{0,\dots,h\}$, there  exists $\nu \in \fh$ such that $\langle \eta_i, \nu \rangle > 0$ for all $i \in I$ exactly if either $p$ is short or   $\{\eta_i\}_{i \in I}$ is linearly independent.

\end{lemma}

\begin{proof}
By Remark~\ref{rem:xii}, there exists
a non-zero $x \in \R^I_{\geq 0}$ such that $\sum_{i \in I} x_i \eta_i = 0$ exactly if  $p$ is tall and $\{\eta_i\}_{i \in I}$
is linearly dependent.  
Hence, it is enough to prove that  there exists $\nu \in \fh$ such that $\langle \eta_i, \nu \rangle > 0$ exactly if 
$\sum_{i \in I} x_i \eta_i \neq 0$ for all non-zero  $x \in \R^I_{\geq 0}$.

Assume first that 
there exists $\nu \in \fh$ such that $\langle \eta_i,\nu \rangle > 0$ for all $i \in I$.
If $x \in \R_{\geq 0}^I$ is non-zero then $\sum_{i \in I} x_i \eta_i \neq 0$ because
$$ \bigg\langle  \sum_{i \in I} x_i \eta_i ,\nu \bigg\rangle = \sum_{i \in I} x_i \langle \eta_i,  \nu \rangle > 0.$$

Conversely, assume that $\sum_{i \in I} x_i \eta_i \neq 0$ for all non-zero $x \in \R^I_{\geq 0}$.
Then $\eta_i \neq 0$ for all $i$ and the closed cone
$$\overline{C} := \sum_{i \in I} \R_{\geq 0} \eta_i$$
is pointed, that is,  $\overline{C} \cap - \overline{C} = \{0\}$.
Consider the dual cone to $\overline{C}$ in ${{\fh}^{\,**}}=\fh$: $$ \overline{C}^{\, *} =
 \{\mu \in \fh \,|\, \langle \alpha, \mu \rangle \geq 0 \text{ for all }\alpha \in \overline{C}\}.$$
By \cite[Corollary 11.7.1]{Rock}, which is a corollary of the hyperplane separation theorem, the fact that  $\overline{C}$ is a non-empty closed convex cone implies that
$\overline{C} = \overline{C}^{\, **},$ the dual cone to $\overline{C}^*.$ 
Therefore, $\overline{C}^{**}$ is  pointed, and so $\overline{C}^*$ is not contained in a hyperplane through $0$, that is, $\overline{C}^*$ is full-dimensional. 
Hence, by \cite[Theorem 6.2]{Rock},  there exists some $\nu$ in the interior
of $\overline{C}^{\,*}$.
It is straightforward to check that this implies  $\langle \alpha,\nu \rangle > 0$ for any
nonzero $\alpha \in \overline{C}$; in particular,
$\langle \eta_i,\nu \rangle > 0$ for all $i \in I$. 
 \end{proof}

Finally,  the following criterion for a point to be exceptional 
is an immediate consequence of the symplectic local normal form and \cite[Lemma 4.4]{Tall1}.

\begin{lemma} [Karshon-Tolman]\label{fact:exp}
Let $\tham$ be a complexity one Hamiltonian $T$-manifold. A  point $p \in M$ with stabilizer $H$ 
is exceptional exactly if 
the subspace of the symplectic slice fixed by the $H$-action is trivial.
\end{lemma}

\section{Holomorphic  and K\"ahler torus actions} \label{ss:holomorphic}

In this section, we introduce the facts about K\"ahler torus actions (of arbitrary complexity) that we will need in our argument.
We begin with a brief review of holomorphic torus actions.

Recall that, for any Lie subgroup $G \subseteq U(n)$ with Lie algebra $\mathfrak g \subseteq  \mathfrak u(n)$, the complexification of $G$ is (naturally isomorphic to) the subgroup
$$G^\C  := \{\exp( i \zeta) u \mid u \in G, \zeta \in \mathfrak g \} \subseteq GL(n,\C), $$
which obviously contains $G$. In particular, 
 the complexification of the compact torus $(S^1)^d$ is $(\C^\times)^d.$

Assume that our compact torus  $T$ acts on a  complex manifold $M$, preserving the complex structure. If $M$ is compact,
then there exists a unique holomorphic action of $T^\C$ on $M$ extending the $T$-action; see, e.g., \cite[Section 2]{IshidaKarshon}.
In any case, there is at most one such action.
Note that by ``holomorphic action", we include the assumption that the map $T^\C \times M \to M$ is holomorphic.

Let $T^\C$ act holomorphically on a complex manifold $(M,J)$.
Given $p \in M$, let $(T^\C)_p$ be the stabilizer of $p$ in $T^\C$
and let $\mathcal O = T^\C \cdot p$ be the $T^\C$-orbit through $p$.
Then $T_p \mathcal O$ is preserved by $J$ and so the {\bf holomorphic slice} $T_p M/T_p \mathcal O$
inherits a natural complex structure  and a holomorphic $(T^\C)_p$-action,
called the {\bf (holomorphic) slice representation}.

The next  lemma helps determine
which orbits in a holomorphic slice representation contain the origin in their closure.

\begin{lemma}\label{closure} 
Fix a compact abelian group $H$ and  $I \subseteq \{1,\dots,n\}$. Let  $H^\C$ act on the $i$'th coordinate of $\C^I$  with weight $\eta_i \in \fh^*$ for each $i \in I$.
If  $z \in (\C^\times)^I$
then $0 \in \overline{ H^\C \cdot z}$ exactly if there exists $\nu \in \mathfrak h$ such that 
$\langle \eta_i, \nu \rangle > 0$ for all $i \in I$.

\end{lemma}

\begin{proof}
If 
there exists $\nu \in \mathfrak h$ such that $\langle \eta_i,\nu \rangle > 0$ for all $i$, then
$\lim_{t \to \infty} e^{- \sqrt{-1} t  \nu} \cdot z = 0$, and so $0 \in \overline{H^\C \cdot z}.$
Conversely, assume that
for any $\nu \in \mathfrak h$ 
 there exists $i$ such that $\langle \eta_i,\nu  \rangle \leq 0$.  Then for every $w \in H_0^\C \cdot z$ there will be at least one coordinate so that $|w_i| \geq |z_i|$, and so  $0 \not\in \overline{H_0^\C \cdot z}$.  Therefore,  $0 \not\in \overline{ \gamma \cdot H_0^\C \cdot z} = \overline{ H_0^\C \cdot \gamma z}$ for any $\gamma \in H^\C$.  Finally,  since $H$ is compact, $H^\C/H_0^\C \simeq H/H_0$ is finite and so  $0 \not\in \overline{H^\C \cdot z} = \cup_{\gamma \in H^\C/H_0^\C} \overline{ \gamma \cdot H_0^\C \cdot z}$.
\end{proof}

We now restrict our focus to the case that $T$ acts in a Hamiltonian fashion on a K\"ahler manifold, and the action extends to a holomorphic $T^{\C}$-action.
Our first goal is to show that in this case
the symplectic and holomorphic slice representations are isomorphic as complex representations of the $T$-stabilizer $T_p$ at a point $p$.  
To do this, we need the following special case of \cite[Proposition 1.6]{Sjamaar2}.

\begin{lemma}[Sjaamar] \label{fact:stab}
Let $(M,\omega,J,\phi)$ be a K\"ahler Hamiltonian $T$-manifold; assume that the $T$-action extends to a holomorphic $T^\C$-action. Given $p \in M$, the stabilizer $(T^\C)_p$ of $p$ in $T^\C$ is the complexification $(T_p)^\C$ of the stabilizer $T_p$ of $p$ in $T$.
\end{lemma}

\begin{lemma} \label{fact:comp}
Let $(M,\omega,J,\phi)$ be a K\"ahler  Hamiltonian $T$-manifold; assume that the $T$-action extends to a holomorphic $T^\C$-action.
Fix $p \in M$.  As a complex representation, the action of $T_p$
 on the symplectic slice $T_p \mathfrak O^\omega/T_p \mathfrak O$ is  isomorphic
to the restriction  to $T_p$  of the $(T_p)^\C$-action  on the holomorphic slice $T_p M/T_p \mathcal O$.
Here,  $\mathfrak O = T \cdot p$ and $\mathcal O = T^\C \cdot p$ denote the  
 $T$-orbit and $T^\C$-orbit through $p$, respectively. \end{lemma}

\begin{proof}
Since
 $T_p \mathfrak O \subset  T_p \mathcal O$,  
    we have a commutative diagram \begin{equation} \label{diagramrestrict}
\begin{array}{c}
\begin{xymatrix}{
& T_p \mathfrak O^\omega \cap J(T_p \mathfrak O^\omega) \ar[dl]_{}
\ar[dr]^{} & \\
T_p \mathfrak O^\omega/T_p \mathfrak O \ar[rr]^{} & &T_pM/T_p \mathcal O}
\end{xymatrix}
\end{array}
\end{equation}

    Since  $J$ and $\omega$ are compatible, $ T_p \mathfrak O \cap J(T_p \mathfrak O^\omega)  = \{0\}$;   since, additionally, $T_p \mathfrak O + J(T_p \mathfrak O)=T_p \mathcal O$  and $T_p \mathcal O$ is a complex subspace,
    $$T_p \mathcal O \cap {T_p \mathfrak{O}}^\omega \cap J(T_p \mathfrak O^\omega) =  T_p \mathcal O \cap ( T_p \mathfrak O + J(T_p \mathfrak O))^\omega =  T_p \mathcal O \cap (T_p \mathcal O)^\omega = \{0\}.$$
     Thus, the left and right arrows in \eqref{diagramrestrict} are injective.

    Moreover,  each quotient space on the bottom row has (real) dimension $\dim {M}-2 \dim \mathfrak O$, while the subspace at the top is K\"ahler and has dimension at least $\dim M-2 \dim \mathfrak O$.   Therefore, the left arrow is a $T_p$-equivariant isomorphism of symplectic vector spaces,  and the right arrow is a $(T^\C)_p$-equivariant isomorphism of complex vector spaces.   Hence, there is a $T_p$-invariant compatible complex structure on the symplectic slice so that the bottom arrow is a  $T_p$-equivariant complex isomorphism.
 \end{proof}

Consequently, under the natural identification of the weight lattices of $T_p$ and $(T_p)^\C \simeq (T^\C)_p$,
the weights of the action of $T_p$ on the symplectic slice $T_p\mathfrak O^{\omega}/T_{p}\mathfrak O$
agree with the weights of the action of $(T^\C)_p$ on the holomorphic slice $T_p M/T_p \mathcal O$.  Hence, we will freely use the term ``slice representation" and ``isotropy weights" to refer to either the symplectic slice or the holomorphic slice.

The following theorem is a special case of \cite[Theorem 1.12]{Sjamaar2} (see also \cite[Theorem 1.21]{Sjamaar2}); our statement includes details from the proof of that theorem.

\begin{theorem}[Sjamaar]\label{reyer}
Let $(M,\omega,J,\phi)$ be a K\"ahler  Hamiltonian $T$-manifold; assume that the $T$-action extends to a holomorphic $T^\C$-action.
Fix $p \in M$ with $T$-stabilizer $T_p$.  
There exists a $T^\C$-invariant open neighborhood $U$ of the $T^\C$-orbit $\mathcal O$ of $p$ in $M$, a $(T_p)^{\C}$-invariant open neighborhood $S$ of the origin $0$ in the holomorphic slice $T_p M/T_p \mathcal O$, and a biholomorphic $T^\C$-equivariant map $T^\C \times_{(T_p)^{\C}} S \to U$ taking $[1,0,0]$ to $p$.
\end{theorem}

\begin{proof}

The $T$-orbit through $p$ is isotropic. Hence, 
as Sjamaar states in the proof of \cite[Theorem 1.12]{Sjamaar2}, there exists a small open ball $B$ about the origin in $(T_p \mathcal O)^\perp \simeq T_p M / T_p \mathcal O$ such that there is a $T^\C$-equivariant biholomorphism  from $T^\C \times_{(T_p)^{\C}} (T_{p})^{\C} \cdot B$  onto an open neighborhood of $\mathcal O$ in $M$ taking $[1,0,0]$ to $p$.
\end{proof}

 This theorem has the following important consequence.

\begin{lemma} \label{orbitclosure}
Let $(M,\omega,J,\phi)$ be a  compact K\"ahler Hamiltonian $T$-manifold,
 and consider the holomorphic $T^\C$-action extending the $T$-action.
Every $T^\C$-orbit  $\mathcal O$ contains a finite number of $T^\C$-orbits in its closure  $\overline{\mathcal {O}}$.\end{lemma}

\begin{proof}
Since $M$ is compact, it suffices to prove that every point $p$ in $M$ with $T^\C$-stabilizer $(T^\C)_p$ has a $T^\C$-invariant  neighborhood that
contains a finite number of orbits in $\overline{\mathcal O}$.
Hence, by Lemma~\ref{fact:stab} and Theorem~\ref{reyer}, it is enough to 
that each $(T^\C)_p$-orbit  on the symplectic slice at $p$ contains a finite number
of $(T^\C)_p$-orbits in its closure.  
Identify the symplectic slice with $\C^n$
so that $(T^\C)_p$ acts as a subgroup of $(\C^\times)^n$ and fix $z \in \C^n$.
Given $J \subseteq \{1,\dots,n\}$ such that $z_j \neq 0$ for all $j \in J$,
define 
a point $z^J \in (\C^\times)^J :=  \{ w \in \C^n \mid w_i \neq 0 \iff  i \in J \}$  by $$z^J_i =  \begin{cases} z_i & \text{if } i \in J \\ 0 & \text{if } i \not\in J \end{cases}.$$
By looking at coordinates, we see that 
\begin{equation} \label{eq:orbitj}\overline{(T^{\C})_p \cdot z} \cap (C^\times)^J \subseteq \overline{(T^{\C})_p \cdot z^J}.\end{equation}
On the other hand, since $(T^{\C})_p$ is closed, it is rational,
and so its image under the natural projection $(\C^\times)^n \to (\C^\times)^J$
 is rational, and hence closed.  Therefore, $(T^{\C})_p \cdot z^J$ is a closed subset of $(\C^\times)^J$.
By \eqref{eq:orbitj}, this implies that $\overline{(T^{\C})_p \cdot z} \cap (\C^\times)^J$ contains at most one $(T^\C)_p$-orbit.\end{proof}

\section{The set $\Mcore$}\label{sec:mcore}

Let $\tka$ be a   K\"ahler complexity one Hamiltonian $T$-manifold;  assume that the $T$-action extends to a holomorphic $T^\C$-action. 
We define the {\bf core} of $M$, denoted $\Mcore$,  to be the union of all the $T^\C$-orbits whose closures
contain tall exceptional points. 
Our goal in this section is to show that   if $M$ is compact then  the orbital moment map  $\bar{\phi} \colon M/T \to \ft^*$ restricts to a {\bf local homeomorphism} from $(\overline{\Mcore} \cap \Mtall)/T$ to $\phi(\Mtall)$, that is,
for all $p \in \overline{\Mcore} \cap \Mtall$  there
exists a neighborhood $W$ of $[p]$ in $(\overline{\Mcore} \cap \Mtall)/T$ such
that $\bar{\phi} \colon W \to \phi(M_\tall)$ is   a homeomorphism onto an open subset of $\phi(M_\tall)$;  see Proposition~\ref{localhomeob}.
 We begin with an important definition, and then describe some explicit examples.

\begin{definition}
  The {\bf catchment} of  a point $p \in M$,  denoted $M_p$,  is the union of the $T^\C$-orbits whose closure contains $p$, that is,
     $$M_p:=\{m \in M \mid p \in \overline{T^\C \cdot m}\}.$$
\end{definition}
Note that, 
by definition,
    $$\Mcore= \cup_{p \in \Mtall \cap \Mexc} M_p.$$

\begin{example}\label{examp:1}
Let the circle $S^1$ act on $\CP^2$ by $\lambda \cdot [z_0,z_1,z_2] = [\lambda z_0, z_1, \lambda^{-1} z_2]$ with moment map $ \phi([z_0,z_1,z_2])=  \frac{1}{2} \frac{|z_0|^2 - |z_2|^2}{ |z|^2}.$
The action, which extends to a holomorphic $\C^\times$-action,
is free except three fixed points $p_0 := [1,0,0],  p_1 := [0,1,0]$, and $p_2 := [0,0,1]$ and a sphere $N := \{z_1 = 0\}$  that is fixed by $\Z_2$; see Figure \ref{graph-hirz-0}.
Moreover, 
\begin{gather*} M_{p_0} = \{z_0 \neq 0\}, \quad M_{p_1} = \{z_0 = 0 \text{ and } z_1 \neq 0 \} \cup \{z_2 = 0 \text{ and } z_1 \neq 0 \}, \quad M_{p_2} = \{ z_2 \neq 0 \}, \quad \text{and} \\
M_q = N \smallsetminus \{p_0,p_2\} \quad \forall q \in N \smallsetminus \{p_0,p_2\}.
\end{gather*}
Since $p_0$ and $p_2$ are the only short points, and since $\Mexc = N \cup p_1$,
\begin{gather*}  \Mcore = M_{p_1} \cup N \smallsetminus \{p_0,p_2\} =
\{ z_i = 0  \text{ for some } i \} \smallsetminus \{p_0,p_2\}, \\
\overline{\Mcore} = \{ z_i = 0 \text{ for some } i \}, \quad \text{and} \quad \overline{\Mcore} \cap \Mtall = \Mcore.
\end{gather*}
\end{example}

\begin{figure}[ht]
\begin{center}
\begin{tikzpicture}
[
	point/.style = {draw, circle,  font=\tiny, fill = black, inner sep = 1.5pt},
	every edge quotes/.style = {font=\tiny, color = black, auto=left, inner sep = 1pt},
]
\draw node at (0.5,0) [point]{};
   \draw node at (-0.5,1) [point]{};
 \draw node at (0.5,2) [point]{};
	\draw [red, thick, dotted]
	(-0.5,1) edge  (0.5, 2);
	\draw [red, thick, dotted]
	(0.5,0) edge  (-0.5,1);
	\draw[ black, thick, dashed]
(0.5,0) edge (0.5,2);
       \node at (-.7,1.25) {$p_1$};
        \node at (.25,2.15) {$p_0$}; 
        \node at (.8,.2) {$p_2$};
        \end{tikzpicture}
\caption{The fixed points of the $S^1$-action in Example \ref{examp:1} and their moment images. The set $\overline{\Mcore}$ is indicated by the dashed and dotted lines; the dashed black line corresponds to the sphere $N$ that is fixed by $\Z_2$. }\label{graph-hirz-0}
\end{center}
\end{figure}
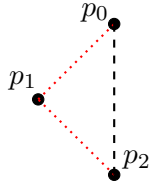

\begin{example}\label{examp:2}
The torus $K := (S^1)^2$ acts on  $\C^4$ by $(\alpha, \beta) \cdot z =
(\alpha z_0, \alpha z_1,  \alpha \beta z_2, \beta z_3)$ with moment map $\psi(z) =  \frac{1}{2} ( |z_0|^2 + |z_1|^2 + |z_2|^2, |z_2|^2 + |z_3|^2).$  The reduced 
space $M := \psi^{-1}(2, 1)/K$ 
is  a Hirzebruch surface with a K\"ahler structure induced from the natural identification $M \simeq  (\C^2 \smallsetminus \{0\})^2/K^\C$.  The  circle $S^1$ acts on $M$ by $\lambda \cdot [z] = [\lambda z_0, z_1, z_2, z_3]$ with moment map $\phi(z) = 
\frac{1}{2} |z_0|^2$. The action, which naturally extends to a holomorphic $\C^\times$-action,  is free except for one fixed sphere $N := \{z_0 = 0 \}$  and two isolated fixed points
 $p := \{ z_1 = z_3 = 0\}$ and  $q := \{  z_1 = z_2 = 0\}$  with moment images $0$,  $1$, and $2$, respectively; see  Figure \ref{graph-hirz-2}.
Moreover, 
  $$ M_p = \{ z_1 z_3 = 0  \} \smallsetminus (\{q\} \cup N) \quad \text{and} \quad
 M_q = \{ z_0 z_3  \neq 0 \}.
 $$
 Since $q$ is the only short point and $p$ and $q$ are the only  exceptional points, 
$\Mcore = M_p$, and so
$$\overline{\Mcore} = \{ z_1 z_3 = 0\} \quad \text{and} \quad
\overline{\Mcore} \cap \Mtall = \{z_1 z_3 = 0\} \smallsetminus \{q\}.$$

\end{example}

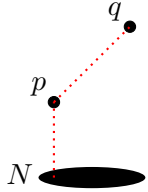
\begin{figure}[ht]
\begin{center}
\begin{tikzpicture}
[
	point/.style = {draw, circle,  font=\tiny, fill = black, inner sep = 1.5pt},
	every edge quotes/.style = {font=\tiny, color = black, auto=left, inner sep = 1pt},
]
   \draw node at (-0.5,1) [point]{};
 \draw node at (0.5,2) [point]{};
 \fill (0,0) ellipse [x radius=20pt, y radius=4pt] node[below=2pt]{};
	\draw [red, thick, dotted]
	(-0.5,1) edge  (0.5, 2);
	\draw [red, thick, dotted]
	(-0.5,0) edge  (-0.5,1);
       \node at (-.7,1.2) {$p$};
        \node at (.3,2.2) {$q$}; 
        \node at (-.95,0.05) {$N$};
        \end{tikzpicture}
\caption{The fixed components of the $S^1$-action in Example \ref{examp:2} and their moment images. The set $\overline{\Mcore}$ is indicated by the dotted lines.}\label{graph-hirz-2}
\end{center}
\end{figure}

It is straightforward to check that the orbital moment map $\overline{\phi}$ does restrict to a local homeomorphism from $(\overline{\Mcore} \cap \Mtall)/S^1$ to $\phi(\Mtall)$ in Examples~\ref{examp:1} and \ref{examp:2}. Note that, in both cases,
the given $S^1$-action is the restriction of a toric $(S^1)^2$-action.

{ Our next goal is to show that every compact K\"ahler complexity one manifold $\tka$ has a finite number of tall exceptional $T^\C$-orbits. (Recall that in the   compact case the $T$-action always extends  to a unique holomorphic $T^\C$-action;
see \S \ref{ss:holomorphic}.)  First, we need to show that we are justified in speaking of ``exceptional" and ``tall" $T^\C$-orbits.

\begin{lemma}\label{lem:exporbit} 
Let $\tka$ be a   K\"ahler complexity one Hamiltonian $T$-manifold;  assume that the $T$-action extends to a holomorphic $T^\C$-action.
If $p \in M$ is tall, then every point in the orbit $T^\C \cdot p$ is tall;
similarly, if $p \in M$ is exceptional, then every point in the orbit $T^\C \cdot p$ is exceptional.
\end{lemma}

\begin{proof}
Every point in $ T^\C \cdot p$ has the same $T^\C$-stabilizer and slice representation. So by Lemma~\ref{fact:comp}, they also have the same $T$-stabilizer and slice representation. 
Therefore, the claim follows immediately from Lemmas~\ref{lem:KT-pos-combo} and \ref{fact:exp}.
\end{proof}

\begin{lemma}\label{finiteexceptional}
Every compact  K\"ahler complexity one Hamiltonian $T$-manifold $\tka$ contains finitely many exceptional $T^\C$-orbits.\end{lemma}
\begin{proof} 

Since the manifold $M$ is compact,
 it is enough to prove that every point $p \in M$ has a neighborhood 
that intersects finitely many exceptional $T^\C$-orbits.
 By Theorem \ref{reyer},  a $T^{\C}$-invariant neighborhood of $p$  is $T^\C$-equivariantly biholomorphic to  a neighborhood of the central orbit in
$$Y = T^{\C} \times_{H^{\C}} {\C}^{h+1},$$
where  $H$ is the $T$-stabilizer of $p$, of dimension $h$,  and $H^\C  \subseteq T^\C$ acts on $\C^{h+1}$ via a faithful representation $H^\C    \hookrightarrow (\C^\times)^{h+1}$.  Given $I \subseteq \{0,\dots,h\}$ define
$$Y_I := \{ [t, z] \in T^\C \times_{H^\C} \C^{h+1}  \mid z_i \neq 0 \iff i \in I \}.$$
Either the induced homomorphism $H^\C \to (\C^\times)^I$ is surjective, or its image
has codimension one.
In the former case, 
the $T^\C$-action on $Y_I$ is transitive. 
In the latter case,
the orbit $T^\C \cdot q$ has codimension one in $Y_I$ for all $q \in Y_I$. Hence, the holomorphic slice 
is $T_q Y_I/T_q (T^\C \cdot q) \simeq  \C$.
Since  every point in $Y_I$ has the same $T^\C$-stabilizer,
this implies that $H^{\C}$ acts trivially on a one-dimensional
subspace of the holomorphic slice  at $q$ in $M$.
Therefore, by Lemmas~\ref{fact:exp} and \ref{fact:comp}, no orbits in $Y_I$ are exceptional.
In either case, $Y_I$ contains at most one exceptional orbit.

\end{proof}
 We now turn to studying the catchment of a tall point $p$. In particular, by Lemma~\ref{local} below,
the catchment $M_p$ contains finitely many $T^\C$-orbits.
In contrast, if $p$ is short then  by Theorem~\ref{reyer}  and Lemmas \ref{lem:nu} and \ref{closure},
the catchment $M_p$ contains infinitely many $T^\C$-orbits.

\begin{lemma}\label{local} 
Let $\tka$ be a   K\"ahler complexity one Hamiltonian $T$-manifold;
assume that the $T$-action extends to a holomorphic $T^\C$-action.
Consider a point $p \in \Mtall$ with stabilizer $H \subseteq T$ and  isotropy weights $\eta_0,\dots  \eta_h \in \fh^*$. Let $\mathcal I$
be the set of subsets $I \subseteq \{0,\dots,h\}$ such  that $\{\eta_i\}_{i \in I}$ is linearly independent. There is a bijection $I \mapsto Y_I$  between $\mathcal I$ 
and the set of $T^\C$-orbits  in the catchment $M_p$.
Moreover, there exists a $T^\C$-invariant neighborhood $V$ of $p$ such that, for each $I \in \mathcal I$:

\begin{enumerate}
\item $\overline{Y_I} \cap V$ is $T^\C$-equivariantly biholomorphic to $T^\C \times_{H^\C} \C^I$,
where  $H$ acts on $\C^I$  with weights $\{\eta_i\}_{i \in I}$.
\item  $\overline{Y_I} \cap V = \cup_{J \subseteq I} Y_J$. 
\end{enumerate}
\end{lemma}

\begin{proof}

Let $H^\C$ act on the $i$'th coordinate of $\C^{h+1}$ with weight $\eta_i$ for each $i$.
By Theorem \ref{reyer} and Lemma~\ref{fact:comp},
 we may assume that, as a complex space with a holomorphic $T^\C$-action,
$M$ is $Y = T^\C \times_{H^\C} \C^{h+1}$ and $p = [1,0,0].$

Given $I \subseteq \{0,\dots,h\}$  define $Y_I := T^\C \times_{H^\C} (\C^\times)^I$, where $(\C^\times)^I := \{z \in \C^{h+1} \mid z_i \neq 0 \iff i \in I\}$. Then $\overline{Y_I} = T^\C \times_{H^\C} \C^I$, where $\C^I := \{z \in \C^{h+1} \mid z_i = 0  \text{ if } i \not\in I \}.$
Fix $[t, z] \in Y_I$. By Lemma~\ref{closure}, the point $p$ is in the closure of the orbit $T^\C \cdot [t,z]$ exactly if there
exists $\nu \in \fh$ such that $\langle \eta_i, \nu \rangle > 0$ for all $i \in I$.
Since $p$ is tall, by Lemma~\ref{lem:nu} this is possible exactly if $\{ \eta_i \}_{i \in I}$ is linearly independant.
In this case, the map $H^\C \to (\C^\times)^I$ is surjective, and so 
 the $T^\C$-action on $Y_I$ is transitive.  The claim follows immediately.
\end{proof}

 Together, the above lemmas  imply that  $\overline{\Mcore}$ also contains
a finite number of $T^{\C}$-orbits.

\begin{lemma}\label{finitecore}
Let $\tka$ be a compact  K\"ahler complexity one Hamiltonian $T$-manifold. 
Then $\overline{\Mcore}$ contains a finite number of $T^\C$-orbits; moreover, $$\overline{\Mcore} = {\cup}_{q \in \Mtall \cap \Mexc} \overline{M_q}.$$
\end{lemma}

\begin{proof}
By Lemma~\ref{finiteexceptional},  $M$ contains finitely many tall exceptional $T^\C$-orbits, and by Lemma~\ref{local} the catchment of each such orbit contains
only a finite number of $T^\C$-orbits.  Hence,  by  Lemma~\ref{orbitclosure},  $\overline{\Mcore}$ contains finitely many $T^{\C}$-orbits.
So every point $p \in \overline{\Mcore}$ must lie in the closure
of some $T^\C$-orbit $Y \subseteq \Mcore$.  By the definition of
 $\Mcore$, there exists a tall exceptional point $q \in \overline{Y}$. Therefore
 $p \in \overline{M_q}$, and so  $\overline{\Mcore} = {\cup}_{q \in \Mtall \cap \Mexc} \overline{M_q}.$

\end{proof}

To prove our main proposition, we will need two key ingredients:
First, we will show  that near any
$p \in \overline{\Mcore} \cap \Mtall$, the set $\overline{\Mcore}$ is equal to $M_p$.   
Next, we will show that the  orbital   moment map $\overline{\phi}$ from 
the quotient $M_p/T$ to the moment cone $\overline{C_p}$ at $p$ is a local homeomorphism at $p$.

\begin{lemma} \label{finite2}
Let $\tka$ be a compact  K\"ahler complexity one Hamiltonian $T$-manifold.
Fix $p \in \overline{\Mcore} \cap \Mtall$.
There exists a $T^\C$-invariant neighborhood $V$ of $p$ such that $\overline{\Mcore} \cap V = M_p$.
\end{lemma}

\begin{proof} 
First, we will show that $M_p:= \{m \in M \mid p \in \overline{ T^\C \cdot m}  \} \subseteq \overline{\Mcore}$.   Let $H$ be the $T$-stabilizer of $p$. If $p$ is exceptional then $M_p \subseteq { \Mcore}\subseteq \overline{ \Mcore}$ by definition.  
So assume that $p$ is not exceptional.    Then by Lemma~\ref{fact:exp}, one of the weights of the isotropy representation at $p$ is zero.  
Since the action is effective, the remaining isotropy weights form a basis for $\mathfrak h^*$,  and so a subset of the isotropy weights is linearly independent exactly if
it is a subset of the set of non-zero weights. Therefore, by Lemma~\ref{local}, there exists a $T^\C$-orbit $Y$ such that  $M_p \subseteq \overline{Y}.$
By Lemma~\ref{finitecore},  there exists a tall exceptional point
$q$ so that $p \in \overline{M_q}$, or equivalently, $q \in \overline{M_p}.$  Hence, $q \in \overline{Y}$, and so $Y \subseteq \Mcore.$ Therefore
$M_p \subseteq \overline{Y} \subset \overline{\Mcore}$.

By definition, if a $T^\C$-orbit lies in $\overline{\Mcore} \smallsetminus M_p$ then its closure also lies in $\overline{\Mcore} \smallsetminus M_p$.
Moreover, by Lemma~\ref{finitecore},  $\overline{\Mcore} \smallsetminus M_p$ contains a finite number of $T^\C$-orbits.
Therefore,  $\overline{\Mcore} \smallsetminus M_p$ is closed,
or equivalently,
 $V:=  (M \smallsetminus {\overline{\Mcore}}) \cup M_p$ is open.
 The claim follows immediately.

\end{proof}

\begin{lemma}\label{homeotall}
Let $\tka$ be a  K\"ahler complexity one Hamiltonian $T$-manifold;  assume that the $T$-action extends to a holomorphic $T^\C$-action.
Given  $p \in \Mtall$, 
there exists a $T$-invariant neighbourhood $U$ of $p$  such that the orbital moment map $\bar\phi \colon M/T \to \ft^*$  restricts to 
 a homeomorphism from $(M_p \cap U)/T$ to an open subset of the moment cone $\overline{C_p}$ of $p$.
\end{lemma}

\begin{proof}
Let $H \subset T$ be the stabilizer of $p$ and  $\eta_0,\dots, \eta_h \in \fh^*$ be the isotropy weights at $p$.
Given $I \subseteq  \{0,\dots,h\}$, consider the polyhedral cone
\begin{equation*} 
\overline{ C_I} :=  \phi(p) + \mathfrak h^\circ +  \sum_{i \in I}  \R_{\geq 0 } \,  \eta_i
\end{equation*} 
with relative interior ${C_I} := \phi(p) + \mathfrak h^\circ +  \sum_{i \in I}  \R_{> 0 } \,  \eta_i.$ 
By Theorem~\ref{le:imagecone}, 
to prove the claim it is enough to find a $T$-invariant neighborhood $U$ of $p$
such that $\barphi$ restricts to injective open map from $(M_p \cap U)/T$ to  the moment
cone  $\overline{C_p} = \overline{ C_{\{0,\dots,h\}}}$.
 Let $\mathcal I$ be the set of $I \subseteq \{0,\dots,h\} $ such that $\{ \eta_i\}_{i \in I}$ is linearly independant.  
By Lemma~\ref{cor:linear}, the moment cone ${\overline{C_{p}}}  = \overline{ C_{\{0,\dots,h\}}}$ is the disjoint union of
 the open cones $\{C_I\}_{I \in \mathcal I}$, that is, 
\begin{equation}\label{eq:Ccup}
 \overline{C_{p}} = \cup_{I \in \mathcal I}\, C_I, \ \  \text{and}
\end{equation}
 \begin{equation}\label{IJempty}
C_I \cap C_J = \emptyset \text{ if } I \neq J \in \mathcal I.
\end{equation}

By Lemma~\ref{local} 
there is a $T^\C$-orbit $Y_I$ for each $I \in \mathcal I$ such that
\begin{equation}\label{eq:mpcup}
M_p  = \cup_{I \in \mathcal I} Y_I. 
\end{equation}Moreover, there exists a $T^\C$-invariant neighborhood $V$ of $p$ such that, for each $I \in \mathcal I$:
\begin{enumerate}
\item  $\overline{Y_I} \cap V$ is $T^\C$-equivariantly biholomorphic to $T^\C \times_{H^\C} \C^I$,
where  $H$ acts on $\C^I$  with weights $\{\eta_i\}_{i \in I}$.
\item  $\overline{Y_I} \cap V = \cup_{J \subseteq I} Y_J$. 
\end{enumerate}
Fix $I \in \mathcal I$. 
Since $\overline{Y_I} \cap V $ is a complex submanifold of a K\"ahler manifold,   it is also a symplectic submanifold.  Let $K_I = \cap_{i \in I} K_i$, where $K_i \subset H$ is the kernal of the homomorphism from $H$ to $S^1$ associated to $\eta_i$. Since $\{\eta_i\}_{i \in I}$ is linearly independent,   $\dim T/K_{I} = \frac{1}{2} \dim \overline{Y_I} \cap V $. 
Consider the induced effective $T/K_I$-action on  $\overline{Y_I} \cap V$. Identifying  $(\mathfrak h/\mathfrak k_I)^*$ with its image in $\mathfrak h^*$, the point $p$   in  $\overline{Y_I} \cap V$ has stabilizer $H/K_I$ and isotropy weights  $\{\eta_i\}_{i \in I}$,
 and so the moment cone of $p$ in $\overline{Y_I} \cap V$ is
$\overline{C_I}$. Moreover,  every point in $Y_I$ has trivial stabilizer.
Therefore,  by Lemma \ref{le:eqdarboux}  there is a $T$-invariant neighborhood $U_I$ of $p$ in $M$ such that  $\phi$ restricts to an open map from  $\overline{Y_I} \cap U_I$  to  $\overline{C_I}$;   moreover, the restriction $\bar{\phi} \colon (\overline{Y_I} \cap U_I)/T \to \overline{C_I}$ is injective and $\phi(Y_I \cap U_I) \subseteq C_I$.
The intersection $U = V \cap (\cap_{I  \in \mathcal I}  U_I) $ is a $T$-invariant open neighborhood of $p$ such that, for all $I \in \mathcal I$,  
\begin{gather}
\label{eq:resphiyi}
 \text{the restriction }\phi \colon \overline{Y_I} \cap U\to \overline{C_I}
  \text{ is open;} \\
\label{eq:indphiyi} \text{the restriction } \bar{\phi} \colon (\overline{Y_I} \cap U)/T \to \overline{C_I}
\text{ is injective;} \\
\label{eq:yicapu}\phi(Y_I \cap U) \subseteq C_I; \text{ and} \\
\label{eq:ybar} \overline{Y_I} \cap U =  \big(\cup_{J \subseteq I} Y_J \big) \cap U.\end{gather}

We claim that  $\phi$ restricts to an open map from  $M_p \cap U$ to $\overline{C_{p}}$.
To see this, fix $q \in M_p \cap U$ and an open neighborhood $N$ of $q$ in $M_p \cap U$. 
Then  $q \in Y_K \cap U$ for some $K \in \mathcal I$ by \eqref{eq:mpcup}.  Fix any $I \in \mathcal I$.
 If $K \subseteq I$ then $q \in \overline{Y_I} \cap U$  by \eqref{eq:ybar}, and so  there is an open neighborhood $W_I$   of $\phi(q)$ in $\ft^* $ such that $W_I   \cap \overline{C_I} \subseteq \phi(\overline{Y_I} \cap N)$ by \eqref{eq:resphiyi}.
 If $K \not\subseteq I$ then since $\overline{C_I} = \cup_{J \subseteq I} C_J$ 
 equations
\eqref{IJempty} and \eqref{eq:yicapu} together imply that $\phi(q) \not\in \overline{C_I}$.
Therefore,  $W_I := \ft^* \smallsetminus \overline{ C_I}$ is an open neighborhood
of $\phi(q)$  such that   $W_I \cap \overline{C_I} = \emptyset \subseteq  \phi(\overline{Y_I} \cap N).$
 The intersection $W:=\cap_{I \in \mathcal I}W_I$ is an open neighborhood of $\phi(q)$ in $\ft^*$
such that $W \cap \overline{C_I} \subseteq \phi(\overline{Y_I} \cap N)$ for all $i \in \mathcal{I}$.
Hence, by  \eqref{eq:Ccup}, \eqref{eq:ybar}, and \eqref{eq:mpcup},
 $$W \cap \overline{C_{p}} = \cup_{I \in \mathcal I} W \cap \overline{C_{I}}\subseteq \phi(\cup_{I \in \mathcal I}\overline{Y_I} \cap N)= \phi(\cup_{I \in \mathcal I} Y_I \cap N) =
\phi(M_p \cap N).$$ 
 The claim follows immediately.

 By the paragraph above,  the orbital moment map
$\barphi$ restricts to an open map from $(M_p \cap U)/T$ to $\overline{C_{p}}$.
Additionally, 
  by \eqref{eq:indphiyi} and $\eqref{eq:yicapu},$ the  map $\bar{\phi}$ restricts to an injection
from $(Y_I \cap U)/T$ to $C_I$  for each $I \in \mathcal I.$  Therefore, $\bar{\phi} \colon (M_p \cap U)/T \to   \overline{C_{p}}$ is an injection by \eqref{IJempty} and  \eqref{eq:mpcup}.The claim follows immediately.

\end{proof}

We now prove our main proposition.

\begin{proposition}
\label{localhomeob}
Let $\tka$ be a compact  K\"ahler complexity one Hamiltonian $T$-manifold.
The orbital moment map $\bar{\phi} \colon M/T \to \ft^*$ restricts to a local homeomorphism
\begin{equation} \label{eq:final3}
    \bar{\phi} \colon {(\overline{\Mcore} \cap \Mtall)}/T \to \phi(\Mtall).
\end{equation}
\end{proposition}

\begin{proof}
By  \cite[Lemma 5.4 and Corollary 6.3]{Tall1}, 
 $\Mtall \subseteq M$ is open. Hence, by Theorem~\ref{le:imagecone},
 it is enough to prove each $p \in  \overline{\Mcore} \cap \Mtall$
 has a $T$-invariant neighborhood $U$ such that the orbital moment map $\overline{\phi}$ restricts
to  a homeomorphism from $(\overline{\Mcore} \cap U)/T$ to  an open subset of  the moment
cone $\overline{C_p}$ of $p$.  But this is an immediate consequence of Lemmas~\ref{finite2} and \ref{homeotall} 
\end{proof}

\section{Proof of Theorem \ref{MainThm}}\label{sec:proof}

In this section, we
prove our main theorem, Theorem~\ref{MainThm}. 
Given a  compact, connected K\"ahler complexity one  Hamiltonian $T$-manifold $(M,\omega,J,\phi)$,  let $\Delta:=\phi(M)$ and $\Delta_{\tall}:=\phi(M_\tall)$. 
In the previous section, we proved that the orbital moment 
map $\overline{\phi}$ restricts to a local homeomorphism
$\bar \phi \colon (\overline{\Mcore} \cap \Mtall)/T \to \Delta_\tall$.
In this section,  we first show that
$\overline{\phi}$ restricts to a (global) homeomorphism
from each connected component of $(\overline{\Mcore} \cap \Mtall)/T$ to $\Delta_\tall$.
Since $\Mexc \cap \Mtall \subseteq \Mcore$,
this will allow us to prove that the painting is trivial.

\begin{proposition}\label{prop:all1}
Let $\tka$ be a  compact, connected K\"ahler complexity one  Hamiltonian $T$-manifold.
Let $A_0$ be a connected component of  $(\overline{\Mcore} \cap \Mtall)/T$.
The  orbital moment map $\bar{\phi} \colon M/T \to \ft^*$ restricts to a homeomorphism from $A_0$
to $ \Delta_\tall$.
\end{proposition}

\begin{proof}
By  Proposition~\ref{localhomeob},
the orbital moment map $\bar \phi \colon M/T \to \ft^*$ restricts to
a local homeomorphism $\bar \phi \colon A_0 \to \Delta_\tall$.
Moreover, since the fibers of $\phi$ are connected, $\phi^{-1}(\Delta_\tall) = M_\tall$.
Since $A_0$ is a closed subset of $\Mtall/T$ and $\phi \colon M \to \ft^*$ is proper, this implies that the restriction $\bar \phi \colon A_0 \to \Delta_\tall$ is proper as a map to its image.
The spaces $A_0 \subseteq M/T$ and  $\Delta_\tall \subseteq \ft^*$ are both Hausdorff.
By Lemma~\ref{deltatall}, $\Delta_\tall$  is simply connected  and locally path connected.
Since $A_0$ is connected and $\bar \phi \colon A_0 \to \Delta_\tall$ is a local homeomorphism,  $A_0$ is path connected.
Therefore, the claim follows from   \cite[Theorem 2]{HoProperMaps}, which states that every proper local homeomorphism from a path-connected Hausdorff space to a simply connected Hausdorff space is a global homeomorphism.
\end{proof}

 We are now ready to prove our main theorem.

\begin{proof}[Proof of Theorem \ref{MainThm}]
Let $\tka$ be a  compact, connected  complexity one K\"ahler Hamiltonian $T$-manifold
of genus $g$.
Let $f \colon (\Mexc \cap \Mtall)/T \rightarrow \Sigma$ be a painting associated to $M$,
 where $\Sigma$ is a closed surface of genus $g$.
 By definition, $f$ is the restriction of a map $F \colon \Mtall/T \rightarrow \Sigma$  that induces 
 a homeomorphism
$$(\bar{\phi}, F) \colon \Mtall/T \longrightarrow \Delta_{\tall}\times \Sigma, $$
and  that preserves the orientation on each reduced space $\phi^{-1}(\alpha)/T$. 

Let $\{A_i\}_{i \in \mathcal I}$ be the set of connected components of $A := (\overline{\Mcore} \cap \Mtall)/T \subseteq M/T$. 
By Proposition \ref{prop:all1}, 
the orbital moment map  $\bar{\phi} \colon M/T \to \ft^*$ restricts to a homeomorphism 
from  $A_i$ to $\Delta_{\tall}$ for all $i \in \mathcal{I}.$
Hence, there exists an inverse map
$s_i \colon \Delta_{\tall}\rightarrow A_i \subseteq A$.
Since  $\Delta_\tall$ is contractible  by Lemma \ref{deltatall}, there exists a deformation retract
	$h_t \colon \Delta_{\tall} \to\Delta_{\tall}, \, \,  t \in [0,1].$
	Define a homotopy
	$H_t \colon A  \to  \Sigma,\,\, t \in [0,1],$
	by 
	$$H_t(a) =F(s_i(h_t(\bar{\phi}(a))))$$
	for  all $t \in [0,1]$, $i \in \mathcal I,$ and $a \in A_i$.
   Then $H_0 =  F\vert_A$, $H_1$ is locally constant, and 
	$$(\bar{\phi}, H_t) \colon A  \longrightarrow \Delta_{\tall}\times \Sigma $$
	is injective for all $t\in [0,1]$.  To  check injectivity of $(\bar{\phi}, H_t)$,  let $a,a'\in A$ be  
	points 
	 such that
	$$\bar{\phi}(a)=\bar{\phi}(a') \quad \text{and} \quad H_t(a)=H_t(a').$$
 There exist $i,j\in \mathcal I$ such that $a\in A_i$ and 
	$a'\in 
	A_j$.   For  $b = s_i(h_t(\bar{\phi}(a))) \in A_i$ and $b' = s_j(h_t(\bar{\phi}(a'))) \in A_j,$  we have
\begin{gather*}  
  \bar{\phi}(b) = h_t(\bar{\phi}(a)) = h_t(\bar{\phi}(a')) = \bar{\phi}(b') \quad \text{and} \\
  F(b)= F(s_i(h_t(\bar{\phi}(a)))=H_t(a) = H_t(a') = F(s_i(h_t(\bar{\phi}(a'))))= F(b').
\end{gather*}
 Since $(\bar{\phi},F)$ is 
	injective, 
	this implies that $b=b'$;  in particular, $i = j$. Since $\bar{\phi}|_{A_i} \colon A_i \to \Delta_{\tall}$ is injective,
 this implies that $a = a'$.

Finally,  by definition, every tall exceptional orbit lies in $\Mcore$, and so $(\Mexc \cap \Mtall)/T \subseteq A$.
Hence, $H_t\vert_{ (\Mexc \cap \Mtall)/T} \to \Sigma$ is a homotopy through paintings
from the given painting $f \colon (\Mexc \cap \Mtall)/T \to \Sigma$ to a locally constant painting. Therefore, $f$ is trivial.
\end{proof}


\begin{thebibliography}{9}

\bibitem{AH}
K. Altmann and J. Hausen,
\emph{Polyhedral divisors and algebraic torus actions},
{Math.\ Ann.} {\bf 334} (2006), no.\ 3, 557--607.

\bibitem{AHS}
K. Altmann, J. Hausen, and H. S\"uss,
\emph{Gluing affine torus actions via divisorial fans},
Transform.\ Groups.\ {\bf 13} (2008), no.\ 2, 215--242.



\bibitem{Atiyah} M.F. Atiyah, 
\emph{Convexity and commuting Hamiltonians}, 
Bull.\ London Math.\ Soc.\ \textbf{14} (1982), no. 1, 1--15.	


	
\bibitem{Delzant}
T. Delzant, \emph{Hamiltoniens p\'eriodiques et images convexes de l'application moment}, Bull.\ Soc.\ Math.\ 
France \textbf{116} (1988), no. 3, 315--339.

\bibitem{Gompf}
R. Gompf, \emph{A New Construction of Symplectic Manifolds},
Ann.\ Math.\
 {\bf 142} (1995), no. 3, 527--595.
	
\bibitem{GSconvex}
V. Guillemin and  S. Sternberg,  \emph{Convexity properties of the moment mapping}, Invent.\ Math.\ \textbf{67} (1982), no. 3, 491--513.
	
	
\bibitem{GSNormalForm} V. Guillemin and S. Sternberg, \emph{A normal form for the moment map}, Differential
geometric methods in mathematical physics (Jerusalem, 1982), Math.\ Phys.\
Stud.\ \textbf{6}, Reidel, Dordrecht, 1984, 161--175.	
	
	
	
	
\bibitem{HoProperMaps} C.-W. Ho, \emph{A note on proper maps}, 
Proc.\ Amer.\ Math.\ Soc.\ \textbf{51} (1975), 237--241.

    
\bibitem{IshidaKarshon} H. Ishida and Y. Karshon, \emph{Completely integrable torus actions on complex manifolds with 
fixed points},	 Math.\ Res.\ Lett.\ \textbf{19} (2012), no. 6, 1283--1295.



	
	\bibitem{Karshon4}
	Y. Karshon,
	\emph{Periodic Hamiltonian flows on four dimensional manifolds},
	Mem.\ Amer.\ Math. Soc.\ \textbf{141} (1999), viii+71.
	

	
	
	\bibitem{Tall1} Y. Karshon and S. Tolman, \emph{Centered complexity one Hamiltonian torus actions},
	Trans.\ Amer.\ Math.\ Soc.\ \textbf{353} (2001), no. 12, 4831--4861.
	
	\bibitem{Tall2} Y. Karshon and S. Tolman, \emph{Complete invariants for Hamiltonian torus actions
		with two dimensional quotients}, J.\ Symplectic Geom.\ \textbf{2} (2003), no. 1, 25--82.
	
	\bibitem{Tall3} Y. Karshon and S. Tolman, \emph{Classification of Hamiltonian torus actions with two dimensional 
		quotients}, Geom.\ Topol.\ \textbf{18} (2014), no. 2, 669--716.
	
	\bibitem{Tall4} Y. Karshon and S. Tolman, \emph{Topology of complexity one quotients}, Pac.\ J.\ Math.\ \textbf{308} (2020), no. 2, 333--346.
	
\bibitem{Lerman} E. Lerman \emph{A compact symmetric symplectic non-Kaehler manifold}, Math.\ Res.\ Lett.\ {\bf 3} (1996), no. 5, 587--590. 	
	
	\bibitem{LMTW}
E. Lerman, E. Meinrenken, S. Tolman, and C. Woodward, 
\emph{Non-abelian convexity by symplectic cuts}, Topology \textbf{37} (1998), no. 2, 245--259.

\bibitem{LPT} Y. Liu, J. Palmer, and S. Tolman,
\emph{Extending complexity one spaces to symplectic toric manifolds},
work in progress.

	\bibitem{Marle} C.M. Marle, 
\emph{Mod\`ele d'action hamiltonienne d'un groupe de Lie sur une vari\'et\'e symplectique},
Rend.\ Sem.\ Mat.\ Univers.\ Politecn.\ Torino
\textbf{43} (1985), no. 2, 227--251. 	

\bibitem{McDuff} D.\ McDuff, {\em Examples of simply-connected symplectic non-K\"ahlarian manifolds}, J.\ Diff.\ Geo.\  {\bf 20} (1984), 267--277.	

\bibitem{Rock}
R.T. Rockafellar, \emph{Convex Analysis}, Princeton University Press, 1970.

   
	
	
	\bibitem{Sjamaar2} 
		R. Sjamaar,
		\emph{Holomorphic slices, symplectic reduction and multiplicities of representations}, Ann.\ Math.\ (2) {\bf 141} (1995), no.\ 1,  
87--129.

\bibitem{Sjamaar} 
R. Sjamaar,
\emph{Convexity properties of the moment mapping re-examined},
Adv.\ Math.\  \textbf{138} (1998), 46--91.




\bibitem{Thurston}
W. Thurston, \emph{Some simple examples of symplectic manifolds}, Proc.\ Amer.\ Math.\ Soc.\ {\bf 55}
(1976), 467--468.

\bibitem{Timashev} 
D. Timashev, 
\emph{Torus actions of complexity one},  
Harada, Karshon, Masuda, and Panov (ed.), Toric topology. 
International conference,
Osaka, Japan, May 28-June 3, 2006. Providence, RI: American
Mathematical Society (AMS). Contemporary Mathematics {\bf 460} (2008),
349--364. 

\bibitem{TolmanExample} S. Tolman, \emph{Examples of non-K\"ahler Hamiltonian torus actions}, Invent.\ Math.\ {\bf 131} 
(1998), 299--310.
	
	
	


\end{thebibliography}
\end{document}